\documentclass[11pt]{amsart}

\usepackage{tikz}
\usetikzlibrary{matrix,calc}
\usepackage{texmac}

\usepackage[margin=3.5cm]{geometry} 
\advance\textheight by 2.5mm

\operators{Stab,Sel,Gal,Irr,Disc,Jac,order,diag}
\calsymbols{c}{C,K,F}

\def\QD{\smash{\widehat Q}}
\def\QDh{\widehat Q}
\def\Cyc{C}
\def\C{{J}}
\def\R{{R}}

\def\step#1#2{\par\vskip 2pt\noindent(#1) \emph{#2}\par\vskip 2pt}

\def\theequation{\arabic{equation}}

\begin{document}

\title{Character formula for conjugacy classes in a coset}

\author{Tim Dokchitser}
\address{Department of Mathematics, University Walk, Bristol BS8 1TW, UK}
\email{tim.dokchitser@bristol.ac.uk} 

\author{Vladimir Dokchitser}
\address{University College London, London WC1H 0AY, UK}
\email{v.dokchitser@ucl.ac.uk}


\begin{abstract}
Let $G$ be a finite group and $N\normal G$ a normal subgroup with $G/N$ abelian.
We show how the conjugacy classes of $G$ in a given coset $qN$ relate to the irreducible characters of $G$ that are not identically 0 on $qN$. 
We describe several consequences. In particular, we deduce that when $G/N$ is cyclic generated by $q$, the number of irreducible characters of $N$ that extend to $G$ is the number of conjugacy classes of $G$ in $qN$.
\end{abstract}

\maketitle

Let $G$ be a finite group and $N\normal G$ a normal subgroup with $Q=G/N$ abelian.
The character group $\QD$ acts on the set of irreducible characters $\Irr G$ by tensoring, and it is well known
that (see e.g. \cite[Thm 1.3]{Tappe})

\begin{center}
  \#(conjugacy classes of $G$ inside $N$) = \#($\QD$-orbits on $\Irr G$).
\end{center}

\noindent
In this note, we give a simple representation-theoretic interpretation of conjugacy classes in 
other cosets of $N$, and discuss some corollaries.
We write $[g]$ for the conjugacy class of $g\in G$, and $[\rho]$ for the $\QD$-orbit
of $\rho\in\Irr G$.

\begin{theorem}
\label{main}
Let $N\normal G$ be finite groups with $Q=G/N$ abelian,
and $q\in G$. Consider

\begin{tabular}{llllll}
     $\C_q$ &=& set of conjugacy classes of $G$ inside $qN$,\Tcr
     $\R_q$ &=& set of $\QD$-orbits $[\rho]$ on $\Irr G$ with $\rho$ not identically 0 on $qN$.\cr
\end{tabular}

\noindent
Then $\# \C_q=\# \R_q$, and the following matrix is unitary:
$$
  M_q=\Bigl(\sqrt{\tfrac{\#[g]\#[\rho]}{\#G}}\,\rho(g)\Bigr)_{[\rho]\in \R_q,\,[g]\in \C_q}.
$$ 
Here we pick any representative $\rho$ for each orbit in $\R_q$.
\end{theorem}

Through the article, a character $\chi$ of a group $G$ containing a normal subgroup $N\normal G$ in
its kernel is sometimes seen as a character of $G/N$ and vice versa. (See e.g. \cite[Lemma 2.22]{Isaacs} 
and the discussion after that.) When $Q$ is abelian, recall that $\hat Q$ is a group under~$\tensor$, and
$\hat Q\iso Q$ non-canonically (see \cite[Problem 2.7]{Isaacs}).

\begin{proof}[Proof of Theorem \ref{main}]
Consider the class functions 
\smash{$
  \pi_q = \tfrac 1{\#\QDh} \sum_{\chi\in\QDh} \overline{\chi(q)} \chi
$}
in the character ring of $G$.
We prove the theorem in 6 steps:

\step{i}{Claim. 
$\pi_q(g)=\leftchoice 0{\text{if $g\notin qN$}}1{\text{if $g\in qN$}}$, and hence
$(\pi_q\tensor\rho)(g)=\leftchoice 0{\text{if $g\notin qN$}}{\rho(g)}{\text{if $g\in qN$}}$.
}

\noindent
Indeed, 
$$
  \pi_q(g) = \tfrac 1{\#\QDh} \sum_{\chi\in\QDh} \overline{\chi(q)} \chi(gN) 
    =\leftchoice 0{\text{if $g\notin qN$}}1{\text{if $g\in qN$}},
$$
by column orthogonality in the character table for $Q$.

\step{ii}{Claim.
If $q\notin\cap_{\chi\in\Stab(\rho)}\ker\chi$, then $\pi_q\tensor\rho=0$.
Otherwise, $\langle \pi_q\tensor\rho,\pi_q\tensor\rho\rangle=\frac{1}{\#[\rho]}$.
}

\noindent
Let $S$ be a set of representatives for $\QDh/\Stab(\rho)$. Every $\chi\in\QDh$ can be written uniquely
as $\chi_1\chi_2$ with $\chi_1\in\Stab\rho$ and $\chi_2\in S$. Then
$$
  \pi_q\tensor\rho = \tfrac 1{\#\QDh} \sum_{\chi\in\QDh} \overline{\chi(q)} (\chi\tensor\rho) =
     \sum_{\chi_2\in S} (\chi_2\tensor\rho) \overline{\chi_2(q)} \sum_{\chi_1\in\Stab(\rho)} \overline{\chi_1(q)}.
$$
If $q\notin\cap_{\chi\in\Stab(\rho)}\ker\chi$, then the inner sum is 0 by column orthogonality for 
$q$ and the identity element in the character table of $\Stab(\rho)$.
Otherwise, it is $\#\Stab(\rho)$, so
$$
  \pi_q\tensor\rho = \tfrac{1}{\#[\rho]} \sum_{\chi_2\in S} \overline{\chi_2(q)} (\chi_2\tensor\rho).
  \eqno{(\dagger)}
$$
In that case, the characters $\chi_2\tensor\rho$ are all distinct, hence orthonormal, and
$$
  \langle \pi_q\tensor\rho, \pi_q\tensor\rho \rangle 
    = \tfrac{1}{(\#[\rho])^2} \sum_{\chi_2\in S} \chi_2(q) \overline{\chi_2(q)} 
    = \tfrac{1}{(\#[\rho])^2} \sum_{\chi_2\in S} 1
    = \tfrac{1}{\#[\rho]}.
$$

\step{iii}{Claim.
$[\rho]\in \R_q \>\iff\> q\in\cap_{\chi\in\Stab(\rho)}\ker\chi$.
}

\noindent
Suppose $[\rho]\in \R_q$, so $\rho\not\equiv 0$ on $qN$. If $\chi\in\Stab\rho$, then 
$\chi\tensor\rho=\rho$, and in particular $\chi(q)=1$ as $\chi$ is constant on $qN$.  
Therefore, $q\in\cap_{\chi\in\Stab(\rho)}\ker\chi$. 
Conversely, if $q$ lies in this intersection, then $\langle \pi_q\tensor\rho,\pi_q\tensor\rho\rangle\ne 0$ by (ii).
As $\pi_q\tensor\rho$ is zero outside $qN$ by (i), we must have $\rho\ne 0$ on $qN$. In other words $[\rho]\in \R_q$.

\step{iv}{Claim.
Choose a set of representatives $U$ of orbits of $\QD$ on $\Irr G$. Then
$$\bigl\{ \sqrt{\#[\rho]}(\pi_q\tensor\rho) \bigm | \rho\in U, q\in\cap_{\chi\in\Stab(\rho)}\ker\chi \bigr\}$$
is an orthonormal basis of class functions for $G$.
}

\noindent
From $(\dagger)$ it is clear that 
$\pi_q\tensor\rho$ and $\pi_{q'}\tensor\rho'$ are orthogonal whenever $\rho\ne\rho'$,
as $[\rho]$ and $[\rho']$ are disjoint.
From (i) it follows that they are orthogonal when $q\ne q'$ as well,
and (ii) shows orthonormality. 

Next, for abelian groups $B\<A$, we have 
$$
  \bigcap_{g\in B}\bigl\{\psi\in\hat A\bigm|\psi(g)=1\bigr\} = \bigl\{\psi\in \hat A \bigm| B\subset\ker\psi \bigr\} = 
  \widehat{A/B}.
$$
Applying this to $B=\Stab \rho$, $A=\smash{\hat Q}$ and $\smash{\hat A}=\smash{\hat{\hat Q}}=Q$ we find that
$$
  \bigcap_{\chi\in\Stab(\rho)}\ker\chi =
  \bigcap_{\chi\in \Stab\rho}\bigl\{q\in Q\bigm|\chi(q)=1\bigr\} %= \bigl\{q\in Q \bigm| \Stab\rho\subset\ker q \bigr\} 
   = \widehat{\hat Q/\Stab\rho}.
$$
In particular, for for each $\rho\in U$, the left-hand side is a group of order $\#Q/\#\Stab\rho=\#[\rho]$. 
%(Generally, for abelian groups $H\<A$, the number of characters in $\widehat A$
%%that are trivial on $H$ is $\#A/\#H$ because these are exactly characters of $A/H$; now apply this to $A=\QDh$, $\widehat A=Q$, 
%and $H=\Stab\rho$.)
Thus our set of class functions has cardinality 
$$
  \sum_{\rho\in U} \#[\rho]=\#\Irr G,
$$
and is therefore a basis.

\step{v}{Claim. 
$\# \C_q=\# \R_q$.
}

\noindent
The class functions $\sqrt{\#[\rho]}(\pi_q\tensor\rho)$ with $q\in\cap_{\chi\in\Stab(\rho)}\ker\chi$ 
are 0 outside $qN$ and are the only such functions from the basis in (iv).
So they are a basis of class functions that are zero outside $qN$, 
and hence their number is the number of conjugacy classes in $qN$.

\step{vi}{Claim. $M_q$ is unitary.}

\noindent
Choose a set of representatives $U_q$ of $\QD$-orbits in $\R_q$. 
By (i), for $\rho, \rho'\in U_q$, we have
$$
  \displaystyle
  \sum_{g\in \C_q}\sqrt{\tfrac{\#[g]\#[\rho]}{\#G}}\,\rho(g) \sqrt{\tfrac{\#[g]\#[\rho']}{\#G}}\,\overline{\rho'(g)} 
  =
  \displaystyle
  \frac{1}{\# G} \sum_{g\in \C_q} \#[g] \sqrt{\#[\rho]}(\pi_q\tensor\rho)(g) \sqrt{\#[\rho']}\, \overline{(\pi_q\tensor\rho')(g)}.
%$}
$$
By (i), $\pi_q\tensor\rho$ is 0 outside $qN$, so this is just 
$$
  \langle \sqrt{\#[\rho]}\pi_q\tensor\rho, \sqrt{\#[\rho']}\pi_q\tensor\rho'\rangle,
$$
which is 1 if $\rho=\rho'$ and 0 otherwise, by (iv).
So the rows of $M_q$ are orthonormal. As $M_q$ is a square matrix by (v), it is unitary.
\end{proof}

%\pagebreak[4]

\begin{example}
Consider $G=F_5=\Cyc_5\rtimes \Cyc_{4}$ of order 20, and $N=\Cyc_5\normal G$. 
Pick $h\in N$ and $q\in G$ of order 5 and 4, respectively.
The character table of $G$ is given on the right.

\noindent
\begin{minipage}{0.65\textwidth}
\ \ \
For the trivial coset $N$ we have 
$$ 
  \C_N=\{[1],[h]\}, \qquad  \R_N=\{[\rho_4],[\rho_5]\}.
$$
Thus, $\#\C_N=\#\R_N$, and indeed 
$$
M_N=\begin{pmatrix} 
  \sqrt{\frac{4}{20}}\cdot 1 & \sqrt{\frac{4\cdot 4}{20}}\cdot 1\cr 
  \sqrt{\frac{1}{20}}\cdot 4 & \sqrt{\frac{4}{20}}\cdot(-1)\cr 
\end{pmatrix}=\frac 1{\sqrt{5}}\begin{pmatrix}  1 & 2 \cr 2 & -1 \end{pmatrix}
$$
is unitary. All other cosets have
$$ 
  \C_{q^jN}=\{[q^j]\}, \quad  \R_{q^jN}=\{[\rho_1]\} \quad\text{and}\quad M_{q^jN}=(1).
$$
\end{minipage}%
\hfill
\begin{minipage}{0.34\textwidth}\hfill
\begin{tikzpicture}
\matrix (M) [matrix of math nodes,row sep=-0.15em, column sep=0em]
{          & \>1 & h & q & q^2 & q^3 \\
  \rm order&\>\rm1&\rm5&\rm4&\rm2&\rm4\\
  \rm size&\>1&4&5&5&5\\
\hline
  \vphantom{\int}
  \rho_{1}&\>1&1&1&1&1\\
  \rho_{2}&\>1&1&-1&1&-1\\
  \rho_{3}&\>1&1&-i&-1&i\\
  \rho_{4}&\>1&1&i&-1&-i\\
  \rho_{5}&\>4&-1&0&0&0\\
};
\draw[black] (M-1-2.north west -| M-8-2.south west) -- (M-8-2.south -| M-8-2.south west);
\draw[black,dashed] (M-7-3.north east -| M-8-3.north east) -- (M-8-3.south east |- M-8-2.south) -- 
     ({$(M-4-1)!.77!(M-4-2)$} |- M-8-2.south) -- ({$(M-4-1)!.77!(M-4-2)$} |- M-7-2.north) -- (M-7-3.north east);
\draw[black,dashed] (M-4-4.north west) -- (M-4-4.north east) -- (M-4-4.south east) -- (M-4-4.south west) -- (M-4-4.north west);
\draw[black,dashed] (M-4-5.north west) -- (M-4-5.north east) -- (M-4-5.south east) -- (M-4-5.south west) -- (M-4-5.north west);
\draw[black,dashed] (M-4-6.north west) -- (M-4-6.north east) -- (M-4-6.south east) -- (M-4-6.south west) -- (M-4-6.north west);
\end{tikzpicture}
\end{minipage}%
\end{example}

\begin{remark}
Generally, Clifford-Fischer theory links cosets of $N$, even when $G/N$ is not abelian, 
to certain projective representations of $G$; see \cite{Fis} and the references in \cite{BM}.
It is quite possible that Theorem \ref{main} is a special case, although usually in Clifford-Fischer
theory to deduce results for ordinary (rather than projective) characters, one requires either that $G=N\rtimes Q$ 
or that every character of $N$ extends to its inertia group. 
Theorem \ref{main} does not need those assumptions (which fail e.g. when $N=\Cyc_2$, $G=Q_8$), and it is probably
easier to prove from first principles in any case.
\end{remark}

\begin{remark}
When $N=G$, the unitary matrix \smash{$M_G=\bigl(\sqrt\frac{\#[g]}{\# G}\rho(g)\bigr)$} is the usual modified character table.
\end{remark}

\begin{lemma}
\label{lemmain}
In the setting of Theorem \ref{main},
\noindent
\begin{enumerate}
\item 
$
  \#\Stab_{\QD}(\rho) = \langle \Res_N\rho,\Res_N\rho \rangle.
$
\item
$[\rho]\in \R_q \>\iff\> q\in\cap_{\chi\in\Stab(\rho)}\ker\chi$.

\item
$\R_q\subset \R_{q^k}$ and $\# \C_q\le \#\C_{q^k}$ for all $q\in G$ and $k\ge 1$.

\item
If $G/N$ is cyclic generated by $q\in G$, then 
$$
  [\rho]\in \R_q \>\>\liff\>\> \Stab_{\QD}(\rho)=\{\triv\} \>\>\liff\>\> \Res_N\rho\text{ is irreducible}.
$$
\end{enumerate}
\end{lemma}

\begin{proof}
(1) Since $Q$ is abelian, $\sum_{\psi\in\QD} \psi$ is the character of the regular representation of~$Q$, in other words 
$\Ind_N^G\triv$ as a character of $G$.
Recall that $\rho\tensor\Ind_N^G\triv=\Ind_N^G((\Res_N\rho)\tensor\triv)$, see e.g. \cite[Problem 5.3]{Isaacs}.
So, by Frobenius reciprocity,
$$
  \#\Stab_{\QD}(\rho) \!=\!
  \langle \rho, \rho\tensor\sum_{\psi\in\QD} \psi \rangle \!=\!
  \langle \rho, \rho\tensor\Ind_N^G\triv \rangle \!=\!
  \langle \rho, \Ind_N^G((\Res_N\rho)\tensor\triv)\rangle \!=\!
  \langle \Res_N\rho,\Res_N\rho \rangle.
$$
(2) This is claim (iii) in the proof of Theorem \ref{main}. 
(3) Immediate from (2) and the equality $\# \R_q=\# \C_q$.
(4) First equivalence follows from (2),
noting that $q\notin\ker\chi$ for any $\triv\ne \chi\in\QDh$. Second equivalence follows from (1). 
\end{proof}

\begin{corollary}
\label{cor1}
Let $N\normal G$ with $G/N$ cyclic generated by $q\in G$. The following are equal:
\begin{itemize}
\item 
The number of conjugacy classes of $G$ inside $qN$.
\item
The number of $\QDh$-orbits $[\rho]$ on $\Irr G$ with $\rho$ not identically 0 on $qN$.
\item
The number of $\QDh$-orbits on $\Irr G$ of length $(G:N)$.
\item
$\tfrac{1}{(G:N)}$ times the number of $\rho\in \Irr G$ whose restriction to $N$ is irreducible.
\item
The number of $\tau\in \Irr N$ that extend to a character of $G$.
\end{itemize}
\end{corollary}

\begin{proof}
The first equivalence follows from Theorem \ref{main} ($\#\R_q=\#\C_q$). 
The second and third are the two equivalences in Lemma \ref{lemmain} (2). 
For the last one, observe that $\tau\in\Irr N$ extends to a character of $G$ if and only if $\tau=\Res\rho$ 
for some $\rho\in\Irr G$. Suppose $\tau=\Res\rho$. Then
$$
  \Res_N \rho' \!=\! \tau \>\Rightarrow\> 1 \!=\! \langle \Res_N \rho', \tau \rangle \!=\!
    \langle \rho', \Ind_N^G \Res\rho \rangle \!=\!
    \langle \rho', \rho\tensor\Ind_N^G\triv_N \rangle 
    \>\Rightarrow\> \rho'\in[\rho].
$$
Conversely, if $\rho'\in[\rho]$, then clearly $\tau=\Res\rho'$. In other words, the characters that restrict
to $\tau$ are exactly those in $[\rho]$, so there are $(G:N)$ of them. The last equivalence now follows.
\end{proof}

\begin{corollary}
Let $N\normal G$ with $G/N$ cyclic. Then $N$ has a non-trivial irreducible character that extends 
to $G$ if and only if $G$ has no conjugacy classes of size $\#N$.
\end{corollary}

\begin{proof}
Let $q\in G$ generate $G/N$. Then $G$ has a conjugacy class of size $\# N$ if and only if $qN$ is such a class,
by Lemma \ref{lemmain} (3). Equivalently, by Corollary \ref{cor1}, only $\triv_N$ extends to a character of $G$.
\end{proof}

\begin{example}
For $p>2$ there are exactly $p$ irreducible representations of $N=\SL_2(\F_p)$ that extend to $G=\GL_2(\F_p)$. 
Indeed, fix a generator $qN\in G/N\iso \F_p^\times$, a primitive root mod $p$. 
Note that if the determinant of a matrix in $\GL_2(\F_p)$ generates $\F_p^\times$, then the element is automatically semisimple, provided $p>2$.
Thus there are $p$ conjugacy classes in the coset $qN$, characterised by their trace, and hence, by Corollary~\ref{cor1}, exactly $p$ irreducible representations that extend to $\GL_2(\F_p)$.
\end{example}

We end with an `inversion formula', which reconstructs a character $\Theta$ on $G$ from its values
on a sufficient number of cosets. This was our original motivation in the number-theoretic setting
when $G$ is a local Galois group and $N\normal G$ its inertia subgroup. In that case,
this formula explictly reconstructs a representation of $G$ from
characteristic polynomials of Frobenius over sufficiently many intermediate fields. We refer the reader to 
\cite{weil2} 
%
%for a summary of this approach, and to our forthcoming paper 
for the applications of the formula.

\begin{corollary}[Inversion formula]
Suppose $N\normal G$ with $Q=G/N$ cyclic, generated by $q\in G$.
Let $U$ be a set of representatives of orbits of $\QDh$ on $\Irr G$, and 
denote $m_\rho=\langle \Res_N\rho, \Res_N\rho\rangle$ for $\rho\in U$.
Let $\Theta$ be a character of $G$.
\begin{itemize} 
\item[(i)]
$\Theta$ can be written as 
$$
  \Theta=\sum_{\rho\in U} \Psi_\rho\tensor\rho \qquad\text{for some character}\>\>\Psi_\rho\>\text{of $G$ with}\>\> N\subset\ker\Psi_\rho.
$$
\item[(ii)]
The eigenvalues of the matrix associated to $\Psi_\rho(q)$ are well-defined up to multiplication by
$m_\rho$th roots of~1.
\item[(iii)]
For every $\rho\in U$ and every $d\ge 0$,
$$
 \Psi_\rho(q^{d m_\rho})=\frac{1}{\# N\!\cdot\! m_\rho} \sum_{g\in q^{d m_\rho} N}  \overline{\rho(g)}\,\Theta(g).
$$
\item[(iv)] 
%If $\rho$ is Frobenius-semisimple, 
$\Psi_\rho\tensor\rho$ is uniquely determined by {\rm (iii)}. Concretely, suppose
\hbox{$\dim\Psi_\rho\!\le\! B$}. There is a unique $0\le n\le B$ and $\lambda_1,...,\lambda_n\in{\mathbb C}^\times$ 
such that 
$\sum_k\lambda_k^d=\Psi_\rho(q^{dm_\rho})$ for $d=1,...,B$.
Then 
%$\Psi_i(\Frob_K)=\diag(\lambda_1^{1/m_\rho},...,\lambda_n^{1/m_\rho})$ or, alternatively,
$\Psi_\rho(q)=\sqrt[m_\rho]{\lambda_1}+...+\sqrt[m_\rho]{\lambda_n}$
for some choice of the roots. 
%
%Conversely, every choice gives a valid $\Psi_i$, and 
The character $\Psi_\rho\tensor\rho$ is independent of this choice.
\end{itemize}
\end{corollary}

\begin{proof}
(i) Clear.
(ii) As the decomposition into irreducibles is unique, the terms $\Psi_\rho\tensor\rho$ are 
uniquely determined by $\Theta$. It remains to show that for 
$\psi,\psi'\in \QD$,
$$
  \rho\tensor\psi\iso \rho\tensor\psi' \quad \liff \quad  (\psi/\psi')^{m_\rho}=\triv.
$$
By Lemma \ref{lemmain} (1), we have $m_\rho=\#\Stab_{\QD}(\rho)$.
Because $\QD$ is cyclic, $\Stab_{\QD}(\rho)$ is exactly the subgroup of characters of
order $m_\rho$, as required.

(iii)
Fix $\rho\in U$ and a multiple $j=dm_\rho$. 
Let $U_{q^j}\subset U$ be the subset of characters that are not identically 0 on $q^j N$.
Define the matrix $M_{q^j}$ as in the theorem 
with these representatives for $R_{q^j}$.
For $g\in \C_{q^j}$, we have
$$ 
  \Theta(g)
    = \sum_{\rho\in U} \Psi_\rho(g)\rho(g)
    = \sum_{\rho\in U} \Psi_\rho(q^j) \rho(g)
    = \sum_{\rho\in U_{q^j}} \Psi_\rho(q^j) \rho(g)
    = \sqrt{\tfrac{\#G}{\#[g]}} \sum_{\rho\in U_{q^j}}  
      M_{q^j,g,\rho} \,
      \frac{\Psi_\rho(q^j)}{\sqrt{\#[\rho]}}.
$$
In other words, we have a matrix equation
$$
  \Bigl (\sqrt{\tfrac{\#[g]}{\#G}}\, \Theta(g) \Bigr)_{g\in \C_{q^j}} = 
    M_{q^j} \cdot \Bigl( \frac{\Psi_\rho(q^j)}{\sqrt{\#[\rho]}} \Bigr)_{\rho\in U_{q^j}}.
$$
As $M_{q^j}$ is unitary, we can rewrite it as 
$$
  \Bigl( \frac{\Psi_\rho(q^j)}{\sqrt{\#[\rho]}} \Bigr)_{\rho\in U_{q^j}} = 
  \overline{M_{q^j}^t} \cdot \Bigl (\sqrt{\tfrac{\#[g]}{\#G}}\, \Theta(g) \Bigr)_{g\in \C_{q^j}}.
$$
Using the fact that $m_\rho=\#\Stab_{\QD}(\rho)$ proved in (ii), 
and the orbit-stabiliser equality $m_\rho [\rho] = (G:N)$, we get
$$
  \Psi_\rho(q^j) = \sqrt{\#[\rho]} \sum_{g\in \C_{q^j}} \overline{M_{q^j,g,\rho}} \sqrt{\tfrac{\#[g]}{\#G}}\, \Theta(g)
    = \sum_{g\in \C_{q^j}} \tfrac{\#[\rho]\#[g]}{\#G}  \overline{\rho(g)} \Theta(g)
$$
$$
\qquad \quad 
    = \sum_{g\in \C_{q^j}} \frac{\#G}{\#N m_\rho} \frac{\#[g]}{\#G}  \overline{\rho(g)} \Theta(g)
    = \frac{1}{\#N m_\rho} \sum_{g\in q^j N} \overline{\rho(g)} \Theta(g),
$$
as claimed.

(iv) Let $\mu_1,...,\mu_n$ be the eigenvalues of the matrix associated to $\Psi_\rho(q)$, 
so that $\Psi_\rho(q^{dm_\rho})=\sum_k\mu_k^{d m_\rho}$ for any $d\ge 0$. 
Generally, the Vandermonde system of equations $\sum_{k=1}^B\nu_k^d=a_d$ for $d=1,...,B$ 
has a unique (unordered) solution $\nu_1,...,\nu_B$.
%, which must therefore be
%$\mu_1^{m_\rho},...,\mu_n^{m_\rho},0,...,0$. 
Thus, the unique solution to $\sum_k\lambda_k^d=\Psi_\rho(q^{dm_\rho})$ is
$\mu_1^{m_\rho},...,\mu_n^{m_\rho},0,...,0$. 
The formula for $\Psi(\rho)$ follows, and independence of the choice is proved in (ii).
\end{proof}

\subsection*{Acknowledgements}
We would like to thank the Warwick Mathematics Institute and King's College London, where parts 
of this research were carried out, Geoffrey Robinson for helpful discussions, and the referee for helpful comments.
This research was partially is supported by EPSRC grants EP/M016838/1 and EP/M016846/1 
`Arithmetic of hyperelliptic curves'. The second author was supported by a Royal Society 
University Research Fellowship.

%%%%%%%%%%%%%%%%
% BIBLIOGRAPHY %
%%%%%%%%%%%%%%%%

\end{document}